\documentclass[12pt]{amsart}

\usepackage{amssymb,amsmath,amsthm,hyperref}
\usepackage{graphicx,color}
\usepackage[all]{xy}
\numberwithin{equation}{section}

\theoremstyle{plain}
\newtheorem{Theorem}{Theorem}[section]
\newtheorem{Corollary}[Theorem]{Corollary}
\newtheorem{Lemma}[Theorem]{Lemma}
\newtheorem{Proposition}[Theorem]{Proposition}

\theoremstyle{definition}
\newtheorem{Definition}[Theorem]{Definition}
\newtheorem{Example}[Theorem]{Example}
\theoremstyle{Remark}
\newtheorem{Remark}[Theorem]{Remark}

\newcommand{\lar}{\longrightarrow}

\def\KK{{\mathbb C}}
\def\RR{{\mathbb R}}
\def\CC{{\mathbb C}}

\def\ker{{\rm ker}\,}

\def\O{{\mathcal O}}

\def\cl#1{{\mathcal #1}}

\def\VaVa{{\cl V}\kern-5pt {\cl V}}

\def\swar#1{\swarrow
	
	\rlap{$\vcenter{\hbox{$\scriptstyle#1$}}$}}

\def\sear#1{\searrow
	
	\rlap{$\vcenter{\hbox{$\scriptstyle#1$}}$}}

\def\pp{{\mathbb P}}

\newcommand{\red}{\color{red}}
\begin{document}
\author{Giorgio Ottaviani}
\address{Dipartimento di Matematica e Informatica ``Ulisse Dini'',  University of Florence, viale Morgagni 67/A, I-50134, Florence, Italy}
\email{giorgio.ottaviani@unifi.it}
\email{zahra.shahidi@unifi.it}

\author{Zahra Shahidi}

\subjclass[2000]{14N07, 14N05, 14N10, 15A69, 15A18}
\keywords{eigenvectors, tensors, singular tuples, vector bundles, Chern classes}

\title{Tensors with eigenvectors in a given subspace}
\maketitle
\begin{abstract}
The first author with B. Sturmfels studied in \cite{OS} the variety of matrices with eigenvectors in a given linear subspace,
called Kalman variety.
We extend that study from matrices to symmetric tensors, proving in the tensor setting the irreducibility of the Kalman variety and computing its codimension and degree.
Furthermore we consider the Kalman variety of tensors having singular t-ples with the first component in a given linear subspace
and we prove analogous results, which are new even in the case of matrices. 
Main techniques come from Algebraic Geometry, using Chern classes for enumerative computations.
\end{abstract}

\section{Introduction}
We introduce the subject of this paper with a basic example. Let $L\subset\KK^n$ be a linear space of dimension $d$.
We are interested in the variety ${\mathcal K}_{d,n,m}(L)$ of matrices $A\in\KK^n\otimes\KK^m$ having a singular pair $(v,w)\in\KK^n\times\KK^m$
with $v\in L$. Namely $A\in {\mathcal K}_{d,n,m}(L)$ if and only if there exists $v\in \KK^n\setminus\{0\}, w\in \KK^m\setminus\{0\}, \lambda_1, \lambda_2\in\KK$ such that $Aw=\lambda_1 v$, $A^tv=\lambda_2 w$ and $v\in L$
(over $\RR$ it is well known that the equality $\lambda_1=\lambda_2$ can be assumed but already on $\CC$
the situation is more subtle, see \cite{CH} and Example \ref{exa:lambdaequal}, note that we have not conjugated the transpose matrix $A^t$). The main Theorem of this paper applies in this case (see Theorem \ref{thm:degree} and Remark \ref{rem:intro}) and it shows that the variety ${\mathcal K}_{d,n,m}(L)$ is algebraic, irreducible, has codimension $n-d$ and it has degree (see Remark \ref{rem:intro})
\begin{equation}\label{eq:degpairs}\sum_{j=0}^{d-1}\sum_{k=d-j-1}^{\min(n-j-1,m-1)}{{n-j-1}\choose k}{k\choose{d-1-j}}.
\end{equation}
For $n\le m$ the expression (\ref{eq:degpairs}) simplifies to
 $2^{n-d}{n\choose{d-1}}$ (see Proposition \ref{prop:mat}), which does not depend on $m$, 
 in other words it stabilizes for $m\gg 0$. Moreover, for $n\le m$, we have that $A\in  {\mathcal K}_{d,n,m}(L)$ 
if and only if the symmetric matrix $AA^t$ has an eigenvector in $L$ (see Lemma \ref{lem:eig}), so that the result is a consequence of 
\cite[Prop. 1.2]{OS}. The paper \cite{OS}, by the first author and B. Sturmfels, was indeed the main source for this paper.
When $n>m$ the condition that $AA^t$ has an eigenvector on $L$ is necessary for $A\in  {\mathcal K}_{d,n,m}(L)$ 
but no more sufficient (see Example \ref{exa:necass}) and \cite[Prop. 1.2]{OS} cannot be invoked anymore.

The formula (\ref{eq:degpairs})  can be expressed as the coefficient of the monomial $$h^{n-d}v^{d-1}w^{m-1}$$ in the polynomial $$\frac{(w+h)^n-v^n}{w+h-v}\cdot\frac{(v+h)^m-w^m}{v+h-w}.$$
Note that both fractions are indeed homogeneous polynomials, respectively of degree $n-1$ and $m-1$. We list the degree
values for small $d, n, m$ in Table 1 in \S 3. The above formulation generalizes to tensors and it  reminds the way we have computed it 
and also
 the formula to compute the number of singular t-ples of a tensors obtained by S. Friedland and the first author in \cite{FO}.
This was interpreted as EDdegree of the Segre variety in \cite{DHOST}.
The techniques we use comes from Algebraic Geometry and corresponds to a Chern class computation of a certain vector bundle.

Singular pairs replace eigenvectors for non symmetric matrices in optimization questions.
Generalizing from matrices to tensors, we have two analogous concepts.
For symmetric tensors in ${\mathrm Sym}^k\CC^n$, eigenvectors \cite{L,Q} are used in best rank one approximation problems,
while for general tensors in $\CC^{n_1}\otimes\ldots\otimes\CC^{n_k}$ $k$-ples of singular vectors\cite{L} play the same role.
Here are our main results.
\begin{Theorem} Let $L\subset \KK^n$ , $\dim L=d$. The variety 
$$\kappa^s_{d,n,k}(L)=\{f\in  \pp(Sym^{k}(\KK^n)) | f\textrm{\ has an eigenvector in\ } L \}$$ is irreducible,  it has codimension $n-d$
and degree $\Sigma_{i=0}^{d-1}\binom{n-d+i}{i}(k-1)^i$  .
\end{Theorem}
\begin{Theorem} Let $L\subset \KK^{n_1}$ , $\dim L=d$. The variety  $$ \kappa_{d,n_1,\ldots, n_k}(L)=\{T \in  \pp(\KK^{n_1}\otimes\ldots\otimes \KK^{n_k}) | 
 T\textrm{\ has  a singular\ }k\textrm{-tuple\ } (v_1, \ldots, v_k)\textrm{\ with\ }v_1\in L \}$$
 is irreducible, it has codimension $n_{1}- d$ and degree given by the coefficient of the monomial $h^{n_1-d}v_{1}^{d-1}\prod_{i\geq 2}v_{i}^{(n_{i}-1)}$ in the polynomial
	$$\prod_{i=1}^{k}\frac{[(\widetilde{v_i}+h)^{n_{i}}-v_i^{n_i}]}{(\widetilde{v_i}+h)-v_i}$$
	where $\widetilde{v_i}=\Sigma_{j}v_{j}-v_i.$.
\end{Theorem}

The structure of this paper is as follows. In \S 2 we study Kalman varieties for symmetric tensors. 
In Proposition \ref{prop:symirred} we prove irreducibility and compute the dimension, in Theorem \ref{thm:degsym}
we compute the degree.
In \S 3 we study Kalman varieties for general tensors, in Proposition \ref{prop:dimtensors} we prove irreducibility and compute the dimension, we compute the degree in Theorem \ref{thm:degree}
and we prove a stabilization result when $n_k\gg 0$, off the boundary format, see
Corollary \ref{cor:stab}.

We thank Luca Sodomaco for helpful discussions.

The first author is member of GNSAGA-INDAM. The second author thanks the Department of Mathematics and Computer Science of the University of Florence for support and hospitality. Both authors acknowledge support from the H2020-MSCA-ITN-2018 project POEMA.

\section{Symmetric tensors with eigenvectors in a given subspace}
Let $L\subset \KK^n$ be a linear subspace of dimension $d$.  A symmetric tensor $ T \in Sym^k(\KK^n)$ can be represented by a homogeneous polynomial $ f_T \in \KK[x_1, \ldots, x_n] $ of degree $k$ given by
$$ f_{T}(x_1, \ldots, x_n)= T\cdot x^k:= \Sigma_{i_1,\ldots,i_k=1}^{n} T_{i_1, \ldots, i_k}x_{i_1}x_{i_2} \ldots x_{i_k}.$$

Eigenvectors of matrices was extended to symmetric tensors by Lim\cite{L} and Qi\cite{Q}.
A vector $v \in V$ is an eigenvector of $T$  if there exist $\lambda \in \KK
$ such that
$$ Tv^{k-1}:= [\Sigma_{i_2, \ldots, i_k=1}^{n}T_{i, i_2, \ldots,i_k}v_{i_2}\ldots v_{i_k}]_{i}= \lambda v,$$
this is equivalent to $ \bigtriangledown f_T(v)= k\cdot \lambda v $, i.e. $\bigtriangledown f_T(v)$ and $v$ are dependent. 
This is the definition used in \cite{CS} and \cite{OO}, while in \cite{QZ} $E$-eigenvectors are defined with the additional requirement that are not isotropic.
The general tensor has no isotropic eigentensors \cite[Lemma 4.2]{DLOT}, so that for general tensors the two definitions coincide.
 We are interested in the scheme $ \kappa^s_{d,n,k}(L)$ of all symmetric tensors $f \in Sym^k\KK^n$ that have an eigenvector in $L$.
\begin{Proposition}\label{prop:symirred}
	The variety $\kappa^s_{d,n,k}(L)$ is irreducible,  it has codimension $n-d$ .
	\begin{proof}
	We regard vectors $v \in \KK^n \setminus \{0\}$	as points in the projective space $\pp(\KK^n)$, and we regard a polynomial of degree $k$ as a point in $ \pp(Sym^{k}(\KK^n))$. The product of these two projective spaces,  $ X= \pp(Sym^{k}(\KK^n))  \times \pp(\KK^n)  $, has the two projections
	\begin{equation}	
	\label{eq:twoprojections}	
	\begin{array}{ccccc}&&X\\	
	&\swar{p}&&\sear{q}\\	
	\pp(Sym^{k}(\KK^n))&&&&\pp(\KK^n)	
	\end{array}	
	\end{equation}
Fix the incidence variety $ W= \{ (f, z) \in  X | \text{z is an eigenvector of f}\}$

The projection of $W$ to the second factor 
\begin{align*}
	 q\colon &W \longrightarrow \pp(\KK^n) \\
	& (f, z) \longmapsto z 
\end{align*}
is surjective and every fiber is $ q^{-1}(z)= \{ f \in \pp(Sym^{k}(\KK^n)) | \bigtriangledown f\cdot z=  \lambda \cdot z\}$ and it is the zero scheme of the $ 2 \times 2$ minors of the matrix $[\bigtriangledown f(z)| z] $. Hence it is a linear subspace of codimension $ n-1$ in $\pp(Sym^{k}(\KK^n))$.  

These properties imply that $W$ is irreducible and has codimension $n-1$ in $\pp(Sym^{k}(\KK^n))  \times \pp(\KK^n)$.

The projection of the incidence variety $W$ to the first factor 
\begin{align*}
	  p\colon &W \longrightarrow \pp(Sym^{k}(\KK^n)) \\
	&(f, z) \longmapsto f 
\end{align*}
is surjective and $p^{-1}(f)= \{z \in \pp(\KK^n)| \text{z is an eigenvector of f}\} $. This set is finite for generic $f$ and its number is equal to $  \dfrac{(k-1)^{n}-1}{k-2}$ by \cite{FS} or \cite{CS} (this expression simplifies to $n$ for $k=2$).

We note that $\kappa_{d,n,k}^s(L)= p(W\cap q^{-1}(\pp(L)) ) \subseteq  \pp(Sym^{k}(\KK^n)) .$

The $(d-1)-$dimensional subspace $\pp(L)$ of $\pp(\KK^n)= \pp^{n-1}$ specifies the following diagram:
\begin{equation}	
\label{eq:twoprojections}	
\begin{array}{ccccc}&&W \cap q^{-1}(\pp(L))\\	
&\swar{p}&&\sear{q}\\	
\kappa_{d,n,k}^s(L)&&&&\pp(L)	
\end{array}	
\end{equation}
Each fiber of the map $q$ in above diagram is a linear space of codimension $n-1$ in $\pp(Sym^{k}(\KK^n))$.

This implies that $W \cap q^{-1}(\pp(L)) $ is irreducible and its dimension equals
$$ \dim (\pp(L))+ \dim(\pp(Sym^{k}(\KK^n)))- (n-1)= (d-1)+ \binom{n+k-1}{k}-(n-1)= \binom{n+k-1}{k}- (n-d) .$$
Since the general fiber of the surjection $p$ is finite, the variety $\kappa_{d,n,k}(L)$ is irreducible of the same dimension. Hence $ \kappa_{d,n,k}^s(L)$ has codimension $n-d$ in $\pp^ {\binom{k+n-1}{k}-1}$.
	\end{proof}
\end{Proposition}
In order to compute the degree of $\kappa_{d,n,k}^s(L)$ we construct a vector bundle $E$ on $X$ with a section
vanishing on $W$.
We briefly recall the construction in \cite[\S 3.1]{OO}.

Let $\O(-1)$ be the universal bundle of rank $1$ and $Q$ the quotient bundle of rank $n-1$. They appear in the following exact sequence, with $\O$ the structure sheaf of $\pp^{n-1}$:
\begin{equation}\label{eq:taut}
	0\lar \O(-1)\lar \O \otimes \KK^n\lar Q\lar 0. \quad \quad 
\end{equation}
Tensoring by $\O(k-1)$ we get
\begin{equation*}\label{eq:taut}
	0\lar \O(k-2)\lar \O(k-1) \otimes \KK^n\lar Q(k-1)\lar 0. \quad \quad 
\end{equation*}
The fiber of $Q(k-1)$ at $v \in \KK^n$,
is isomorphic to $ Hom(\langle v^{k-1}\rangle , \dfrac{\KK^n}{\langle v\rangle })$.

Every $f \in Hom({\mathrm Sym}^{k-1}\KK^n,\KK^n)$ induces a section $s_f \in H^0(Q(k-1))$ which corresponds to the composition $\langle v^{k-1}\rangle  \stackrel{i} \lar  {\mathrm Sym}^{k-1}\KK^n \stackrel{f} \lar \KK^n  \stackrel{\pi} \lar  \dfrac{\KK^n}{\langle v\rangle }$ on the fiber of $\langle v\rangle $. The section $s_f$ vanishes on $\langle v\rangle$ if and only if $\pi(f(v^{k-1}))=0$.
In particular, every symmetric tensor $f\in{\mathrm Sym}^k\KK^n$ defines (by contraction) an element
in $Hom({\mathrm Sym}^{k-1}\KK^n,\KK^n)$ and hence a section of $Q(k-1)$ which by abuse of notation
we denote again with $s_f$. This implies 

\begin{Lemma}{\cite[Lemma 3.7 (2)]{OO}}
	For $f \in {\mathrm Sym}^{k}\KK^n$, the section $s_f$ vanishes in $v$ iff $v$ is a eigenvector of $f$. 
\end{Lemma}

We recall that if a vector bundle $E$ of rank $r$ on $X$ has a section vanishing on $Z$, and the codimension of $Z$ is equal to $r$, then the class of $[Z]$ in the degree $r$ component of the Chow ring $A^*(X)=\KK[h,v]/(h^{{{n+k-1}\choose k}},v^n)$ is computed by $[Z]=c_r(E).$ We shall apply this to the following vector bundle on the product variety $X$
$$ E:= p^*\O_{\pp(Sym^{k}(\KK^n))}(1) \otimes q^* Q(k-1).$$ 
Since $H^{0}(\O(1))= Sym^{k}(\KK^n)$ and $H^{0}(Q(k-1))= \Gamma^{k, 1^{n-2}}\KK^n$, by K\"unneth formula we get
$$ H^{0}(E)= Hom(Sym^{k}(\KK^n),\Gamma^{k, 1^{n-2}}\KK^n). $$
 We have a section ${\textit I}\in H^0(E)$ given by the map $M\mapsto s_M$.
The section $\textit{I}$ vanishes exactly at the pairs $(M, z)$ such that  $z$ is eigenvector of $M$,
so we get that the zero locus $Z(I)$ of  ${\textit I}\in H^0(E)$ equals the incidence variety $W$.

  \begin{Theorem}\label{thm:degsym}
  	The degree of  $\kappa^s_{d,n,k}(L)$  equals $\Sigma_{i=0}^{d-1}\binom{n-d+i}{i}(k-1)^i$ in $\pp^ {\binom{n+k-1}{k}-1}$.
  \end{Theorem}
\begin{proof}
 Since $Z(\textit{I})$ has codimension $rkE= n-1$ in $X$, the class of $Z(\textit{I})$ equals the top chern class of $E$(By theorem 4.4). In symbols, $[Z(\textit{I})]=c_{n-1}(E).$  
	
	The desired degree equals
	\begin{equation}\label{eq:**} \deg \kappa^s_{d,n,k}(L)= p^*c_1(\O(1))^{\binom{n-k-1}{k}-(n-d)-1}\cdot c_{(n-1)}(E)\cdot q^*(c_1(\O(1)))^{n-d}. \end{equation}
	The Chern class of $E= p^*\O_{\pp(Sym^{k}(\KK^n))}(1) \otimes q^* Q(k-1)$ decomposes as 
	$$ c_{(n-1)}(E)= \Sigma_{i=0}^{n-1}p^*c_1(\O(1))^{(n-1)-i}.q^*c_i(Q(k-1)).$$
	Hence the equality on the right of (\ref{eq:**}) can be written as
	\begin{eqnarray}
	\nonumber \deg \kappa_{d,n,k}^s(L)&= & p^*c_1(\O(1))^{\binom{n-k-1}{k}-(n-d)-1}\cdot (\Sigma_{i=0}^{n-1}p^*c_1(\O(1))^{(n-1)-i}\cdot q^*c_i(Q(k-1))\cdot q^*c_1(\O(1))^{n-d})\\
	\nonumber & = & \Sigma_{i=0}^{n-1}p^*c_1(\O(1))^{\binom{n-k-1}{k}-1+(d-1)-i} \cdot q^*[c_i(Q(k-1))\cdot c_1(\O(1))^{n-d}]. 
	\end{eqnarray}
All summands are zero except for $i= d-1$, and we remain with $\deg c_{d-1}(Q(k-1))$.
Since $\deg c_i(Q)=1$, ${\mathrm rk}(Q)=n-1$, the result follows from the  formula
$$c_j(E\otimes L)=\sum_{i=0}^{j}{{{\mathrm rk}E-j+i}\choose{i}}c_i(E)c_1(L)^{i}$$
see \cite[\S 1.2]{OSS}.	
 \end{proof}
\def\niente{
{\red \begin{Theorem} The degree of Kalman variety of symmetric tensor having eigenvector $v\in L$ is the coefficient of $h^{n-d}v^{d-1}$ in the expansion of
		$$\frac{[1+(k-1)v+h]^n}{1+(k-2)v+h}$$
better of 
$$\frac{[(k-1)v+h]^n}{(k-2)v+h}$$
	\end{Theorem}
	\begin{proof}
		The degree is equal to $p^*c_1(\O(1))^{\binom{n+k-1}{k}-1-(n-d)}c_{n-1}(E)q^*(c_1(\O(1)))^{n-d}$.
		Denote $h=c_1(\O_{\pp(sym^{k}(\KK^n))})$, $v=c_1(\O_{\pp(\KK^n)})$.
		By the Euler sequence $E$ has the Chern polynomial
		$\frac{[1+(k-1)v+h]^n}{1+(k-2)v+h}$.
		So the degree is the coefficient of $h^{\binom{n+k-1}{k}-1}v^{n-1}$ in the expansion of
		$$h^{\binom{n+k-1}{k}-1-(n-d)}\frac{[1+(k-1)v+h]^n}{1+(k-2)v+h}v^{n-d}.$$
		The thesis follows.
\end{proof}}
}
\begin{Remark}
The result generalizes to any complex vector space $V$ equipped with a symmetric nondegenerate bilinear form. 
The construction works in the setting of $SO(V)$-actions, in particular the space of sections we have considered,
like $H^0(Q(k-1))$, are $SO(V)$-modules. Note that in this setting $V$ is isomorphic to its dual $V^\vee$.
\end{Remark}

\section{Singular vector $k$-ples of tensors}
Let $L\subset\KK^{n_1}$ be a fixed $d-$dimensional linear subspace .
{ \begin{Definition}
	The $k-$ple $(v_1, \ldots , v_k)$ with $v_i\in \KK^{n_i}\setminus\{ 0\}$ is called a singular $k$-ple for a tensor $T\in \KK^{n_1}\otimes\ldots\otimes \KK^{n_k}$ if 
$$ T(v_1, \ldots, \widehat{v_i}, \ldots, v_k)= \lambda_i v_i,  \textrm{for some\ }\lambda_i\in\KK, \quad i= 1, \ldots, k. $$

\end{Definition}}

 We are interested in the Kalman variety $ \kappa_{d,n_1,\ldots, n_k}(L)$ of all  tensors

\noindent $T \in \KK^{n_1}\otimes\ldots\otimes \KK^{n_k}$ that have 
 a singular $k-$tuple $ (v_1, \ldots, v_k)$ with $v_1\in L$.

\begin{Proposition}\label{prop:dimtensors}
	The variety $\kappa_{d,n_1,\ldots, n_k}(L)$ is irreducible and it has codimension $n_{1}- d$.
\end{Proposition}
	\begin{proof}
		The product   $ X= \pp(\KK^{n_1}\otimes\ldots\otimes \KK^{n_k})  \times \pp(\KK^{n_1})\times \ldots\times \pp(\KK^{n_k})  $, has the two projections
		\begin{equation}	
		\label{eq:twoprojections}	
		\begin{array}{ccccc}&&X\\	
		&\swar{p}&&\sear{q}\\	
		\pp(\KK^{n_1} \otimes \ldots \otimes \KK^{n_k})&&&&\pp(\KK^{n_1})\times  \ldots \times\pp(\KK^{n_k})	
		\end{array}	
		\end{equation}
		Fix the variety (we use the same notation of previous section to ease the analogy) $$ {W}= \{ (T, v_1,\ldots, v_k) \in  X | \text{$( v_1,\ldots, v_k)$ is singular k-tuple for T}\}\subseteq X.$$
		
		The projection of $W$ to the second factor 
		\begin{center}
			$  q: W \longrightarrow \pp(\KK^{n_1})\times  \ldots\times \pp(\KK^{n_k}) $\\
			$ (T, v_1,\ldots, v_k) \longmapsto ( v_1,\ldots, v_k) $
		\end{center}
		
	{ 	is surjective and every fiber is $$ q^{-1}( v_1,\ldots, v_k)= \{ T \in \pp(\KK^{n_1})  ) |  T(v_1, \ldots , \widehat{v_i}  , \ldots, v_k)= \lambda_i v_i\}$$  and this is the zero scheme of 
the ideal of $ 2 \times 2$ minors of the $n_i \times 2$ matrix $( T(v_1, \ldots , \widehat{v_i} , \ldots, v_k)|v_i) $. Hence it is a linear subspace of codimension $\Sigma_{i=1}^{k} (n_{i}-1)= \Sigma_{i=1}^{k}n_i-k $ in $\pp(\KK^{n_1}\otimes\ldots\otimes \KK^{n_k})$.}
		
		The projection of the variety $W$ to the first factor 
		\begin{center}
			$  p: W \longrightarrow \pp(\KK^{n_1}\otimes\ldots\otimes \KK^{n_k}) $\\
			$ (T,v_1,\ldots, v_k) \longmapsto T $
		\end{center}
		is surjective and $p^{-1}(T)= \{(v_1,\ldots, v_k) \in \pp(\KK^{n_1})\times  \ldots \times\pp(\KK^{n_k})| \text{$( v_1,\ldots, v_k)$ is singular $k$-ple for T}\} $.  This set is finite for generic $T$.
		
		Consider the map 
		\begin{center}
			$  q_{1}:   \pp(\KK^{n_1})\times  \ldots \times\pp(\KK^{n_k}) \longrightarrow \pp(\KK^{n_1})$
		\end{center}
		
		The Kalman variety has the following description in terms of the above diagram:
		$$  \kappa_{d,n_1,\ldots, n_k}(L)= p(W \cap q^{-1}(q_1^{-1}(\pp(L))))$$
		where $q_{1}^{-1}(\pp(L))=\pp(L)\times \pp(\KK^{n_2}) \times  \ldots \times\pp(\KK^{n_k}).$
		The $(d-1)-$dimensional subspace $\pp(L)$ of $\pp(\KK^{n_1})= \pp^{n_{1}-1}$ specifies the following diagram:
		\begin{equation}	
		\label{eq:twoprojections}	
		\begin{array}{ccccc}&&W \cap q^{-1}(q_{1}^{-1}(\pp(L)))\\	
		&\swar{p}&&\sear{q}\\	
		\kappa_{d,n_1,\ldots, n_k}(L) &&&&\pp(L)\times \pp(\KK^{n_2}) \times  \ldots \pp(\KK^{n_k})	
		\end{array}	
		\end{equation}
		Each fiber of the map $q$ in above diagram is a linear space of codimension $\Sigma_{i=1}^{k} (n_{i}-1)$ in $\pp((\KK^{n_1}\otimes\ldots\otimes \KK^{n_k})$, and we conclude exactly as in the proof of Proposition \ref{prop:symirred}.
		
	\end{proof}
In order to compute the degree of $\kappa_{d,n_1,\ldots, n_k}(L) $ we construct a vector bundle $E$ on $X$ with a section
vanishing exactly on $W$.
We briefly recall the construction in \cite[\S 3]{FO}.

The fiber of $q_{i}^{*}Q_{\pp(\KK^{n_i})}(1,\ldots 1,\underbrace{0}_{i},1\ldots, 1)$ at $x=[v_1\otimes\ldots\otimes v_k]$, is isomorphic to $Hom(\langle v_1\otimes\ldots\otimes\widehat{v_i}\otimes\ldots\otimes v_k\rangle,\dfrac{\KK^{n_i}}{[v_i]})$. Every tensor 
$T\in \KK^{n_1}\otimes\ldots\otimes \KK^{n_k}\simeq{\mathrm Hom}(\KK^{n_1}\otimes\ldots\otimes\widehat{\KK^{n_i}}\otimes\ldots \KK^{n_k},\KK^{n_i})$ induces a section $s_{T} $  of  $q_{i}^{*}Q_{\pp(\KK^{n_i})}(1,\ldots 1,\underbrace{0}_{i},1\ldots, 1)$
which corresponds to the composition $$\langle v_1\otimes\ldots\otimes\widehat{v_i}\otimes\ldots\otimes v_k\rangle  \stackrel{i} \lar  \KK^{n_1}\otimes\ldots\otimes\widehat{\KK^{n_i}}\otimes\ldots \otimes\KK^{n_k} \stackrel{T} \lar \KK^{n_i} \stackrel{\pi} \lar  \dfrac{\KK^{n_i}}{\langle v_i\rangle}$$ on the fiber of $\langle v_1\otimes\ldots\otimes v_k\rangle$. The section $s_T$ vanishes in  $\langle v_1\otimes\ldots\otimes v_k\rangle$ if and only if  $T(v_1, \ldots, \widehat{v_i}, \ldots, v_k)= \lambda_i v_i,$
for some $\lambda_i\in\KK$. 

 Define the vector bundle 
$\varepsilon= \sum_{i=1}^k q_{i}^{*}Q_{\pp(\KK^{n_i})}(1,\ldots 1,\underbrace{0}_{i},1\ldots, 1)$.

 This implies 

\begin{Lemma}{\cite[Lemma 11)]{FO}}
	For $T \in  \KK^{n_1}\otimes\ldots\otimes \KK^{n_k}$, the diagonal section $(s_T,\ldots, s_T)\in H^0(\varepsilon)$ vanishes in $v_1\otimes\ldots\otimes v_k$ iff $(v_1,\ldots, v_k)$ is a singular $k$-ple of $T$. 
\end{Lemma}

Consider the following vector bundle on the product variety $ X= \pp(\KK^{n_1}\otimes\ldots\otimes \KK^{n_k})  \times \pp(\KK^{n_1})\times \ldots \times \pp(\KK^{n_k})  $
$$ E:= p^*\O_{\pp(\KK^{n_1}\otimes\ldots\otimes \KK^{n_k})}(1) \otimes q^* (\varepsilon).$$ 
Hence
$$ H^{0}(X,E)= H^{0}( p^*\O_{\pp( \KK^{n_1}\otimes\ldots\otimes \KK^{n_k})}(1) \otimes H^{0}(q^*\varepsilon )$$
we get
$$ H^{0}(E)= ( \KK^{n_1}\otimes\ldots\otimes \KK^{n_k}) \otimes [ ( \KK^{n_1}\otimes\ldots\otimes \KK^{n_k}) \oplus \ldots \oplus  ( \KK^{n_1}\otimes\ldots\otimes \KK^{n_k})].$$

 We have a section ${\textit I}\in H^0(E)$ given by the diagonal map $T\mapsto (s_T,\ldots, s_T)$.
The section ${\textit I}$ vanishes exactly at $(T, v_1,\ldots, v_k)$ such that  $(v_1,\ldots, v_k)$ is 
a singular $k$-ple of $T$,
so we get that the zero locus $Z(I)$ of  ${\textit I}\in H^0(E)$ equals the incidence variety $W$.

\begin{Theorem}\label{thm:degree} The degree of Kalman variety $\kappa_{d,n_1,\ldots, n_k}(L)$ is the coefficient of $h^{n_1-d}v_{1}^{d-1}\prod_{i\geq 2}v_{i}^{(n_{i}-1)}$ in the polynomial
	$$\prod_{i=1}^{k}\frac{[(\widetilde{v_i}+h)^{n_{i}}-v_i^{n_i}]}{(\widetilde{v_i}+h)-v_i}$$
	where $\widetilde{v_i}=\Sigma_{j}v_{j}-v_i.$
\end{Theorem}
\begin{proof}
	As in the proof of Theorem \ref{thm:degsym}, the degree is equal to  $$p^*c_1(\O(1))^{[\Sigma_{i=1}^{k} (n_{i}-1)- (n_1-d)]}.c_{(\Sigma_{i=1}^{k} (n_{i}-1))}(E).q_{V_1}^*(\O(1))^{n_1-d}.$$
	Denote $h=c_1(\O_{\pp( \KK^{n_1}\otimes\ldots\otimes \KK^{n_k})}(1))$, $v_i=c_1(\O_{\pp(\KK^{n_i})}(1))$,
they are the generators of the Chow ring $A^*(X)=\KK[h,v_1,\ldots, v_k]/\left(h^{1+\sum_i(n_i-1)},v_1^{n_1},\ldots, v_k^{n_k}\right)$.

	$E$ is the direct sum of $k$ summands, by the Euler sequence each of them has Chern polynomial
	$\frac{(1+v_1+\ldots+\widetilde{v_i}+\ldots v_k+h)^{n_i}}{(1+v_1+\ldots-v_i+\ldots v_k+h)}$ for $i= 1, \ldots, k$.
	So the degree is the coefficient of $h^{(\prod n_{i}-1)}v_{1}^{n_{1}-d}\prod_{i}v_i^{n_{i}-1}$ in the expansion of
	$$h^{(\prod n_{i}-1)-(n_1-d)}\frac{(1+v_2+\ldots +v_k+h)^{n_1}}{(1-v_1+v_2+\ldots+v_k+h)}\ldots \frac{(1+v_1+\ldots +v_{k-1}+h)^{n_k}}{(1+v_1+\ldots+v_{k-1}-v_k+h)}v_{1}^{n_1-d}.$$
Hence it is the coefficient of $h^{n_1-d}v_{1}^{d-1}\prod_{i\geq 2}v_{i}^{(n_{i}-1)}$ in the expansion of
$\prod_{i=1}^k\frac{(1+\widetilde{v_i}+h)^{n_i}}{1+\widetilde{v_i}+h-v_i}$.

Since 
\begin{eqnarray}
\nonumber \frac{(1+\widetilde{v_i}+h)^{n_i}}{(1+\widetilde{v_i}+h-v_i)}&=&\frac{(1+\widetilde{v_i}+h)^{n_i-1}}{1-x}; \quad x=\dfrac{v_i}{1+\widetilde{v_i}+h}\\
\nonumber &=&(1+\widetilde{v_i}+h)^{n_i-1}\Sigma_{p=0}^{\infty}x^p=(1+\widetilde{v_i}+h)^{n_i-1}\Sigma_{p=0}^{n_i-1}x^p \\
\nonumber &=&\Sigma_{j=0}^{n_i-1}(1+\widetilde{v_i}+h)^{n-1-j}v^j
\end{eqnarray}
We get that the degree sought is the coefficient of $h^{n_1-d}v_{1}^{d-1}\prod_{i\geq 2}v_{i}^{(n_{i}-1)}$ in the expansion of the polynomial $$\prod_{i=1}^{k}\frac{[(1+\widetilde{v_i}+h)^{n_{i}}-v_i^{n_i}]}{(1+\widetilde{v_i}+h)-v_i}.$$
All terms of the last polynomial have degree $\le\sum_i{(n_i-1)}$ which equals the degree of the monomial $h^{n_1-d}v_{1}^{d-1}\prod_{i\geq 2}v_{i}^{(n_{i}-1)}$, hence it is enough to consider the homogeneous part of top degree.
	The thesis follows.
\end{proof}

The following stabilization phenomenon is similar to the one observed in \cite{OSV}.
{ \begin{Corollary}\label{cor:stab}
	(Stabilization)
Let $(n_k-1)=\sum_{i=1}^{k-1}(n_i-1)$ (boundary format, see \cite{GKZ}). For any $m\ge n_k$ we have
$\kappa_{d,n_1,\ldots, n_k}(L)=\kappa_{d,n_1,\ldots, n_{k-1},m}(L)$
\end{Corollary}}

Denote $V^{N}=\prod_{i=2}^{k-1}v_i^{n_i-1}$. We need to compute the coefficient of $h^{n_1-d}v_{1}^{d-1}V^Nv_k^{m-1}$ in the polynomial
$$\left(\prod_{i=1}^{k-1}\frac{[(\widetilde{v_i}+h)^{n_{i}}-v_i^{n_i}]}{(\widetilde{v_i}+h)-v_i}\right)\frac{[(\widetilde{v_k}+h)^{m}-v_k^{m}]}{(\widetilde{v_k}+h)-v_k}$$

when $m\ge n_k$.
Compare indeed 
the coefficient of  $h^{n_1-d}v_1^{d-1}V^Nv_k^{n_k-1}$ in 
$$\left(\prod_{i=1}^{k-1}\frac{[(\widetilde{v_i}+h)^{n_{i}}-v_i^{n_i}]}{(\widetilde{v_i}+h)-v_i}\right)(\Sigma_{j=0}^{n_k-1}(\widetilde{v_k}+h)^{n_k-1-j}v_k^{j})$$
with
the coefficient of $h^{n_1-d}v_{1}^{d-1}V^Nv_k^{m-1}$ in the polynomial
$$\left(\prod_{i=1}^{k-1}\frac{[(\widetilde{v_i}+h)^{n_{i}}-v_i^{n_i}]}{(\widetilde{v_i}+h)-v_i}\right)(\Sigma_{j=0}^{m-1}(\widetilde{v_k}+h)^{m-1-j}v_k^{j})$$

The first one is the coefficient of 
 $h^{n_1-d}v_1^{d-1}V^Nv_k^{m-1}$ in 
$$\left(\prod_{i=1}^{k-1}\frac{[(\widetilde{v_i}+h)^{n_{i}}-v_i^{n_i}]}{(\widetilde{v_i}+h)-v_i}\right)(\Sigma_{j=0}^{n_k-1}(\widetilde{v_k}+h)^{n_k-1-j}v_k^{j+m-n_k})$$
namely in 

$$\left(\prod_{i=1}^{k-1}\frac{[(\widetilde{v_i}+h)^{n_{i}}-v_i^{n_i}]}{(\widetilde{v_i}+h)-v_i}\right)(\Sigma_{j=m-n_k}^{m-1}(\widetilde{v_k}+h)^{m-1-j}v_k^{j})$$

With $j<m-n_k$ in the last sum we have  

$$\left(\prod_{i=1}^{k-1}\frac{[(\widetilde{v_i}+h)^{n_{i}}-v_i^{n_i}]}{(\widetilde{v_i}+h)-v_i}\right)(\Sigma_{j=0}^{m-n_k-1}(\widetilde{v_k}+h)^{m-1-j}v_k^{j})$$

which has zero coefficient of $h^{n_1-d}v_{1}^{d-1}V^Nv_k^{m-1}$ 
since the maximum exponent of $v_k$ is $\sum_{i=1}^{k-1}(n_i-1)+(m-n_k-1)=(n_k-1)+(m-n_k-1)=m-2$.  Hence the two coefficients are equal.

This concludes the proof of the Corollary.

The following Proposition covers the case of matrices, $k=2$.

\begin{Proposition}\label{prop:mat} The degree of Kalman variety of $n\times m$ matrices for $n\le m$ is
$$\deg \kappa_{d,n,m}(L) = 2^{(n-d)}{n\choose{d-1}}$$
\end{Proposition}

{ \begin{proof} By Corollary $\ref{cor:stab}$ we may assume $n=m$ (square case).
	By Theorem \ref{thm:degree} we need to compute the coefficient of $h^{n-d}v^{d-1}w^{n-1}$ in the expansion of
	 the polynomial
$$ (\Sigma_{j=0}^{n-1}(w+h)^{n-1-j}v^j)(\Sigma_{j=0}^{n-1}(v+h)^{n-1-j}w^j)$$
By Newton expansion, this is equal to
\begin{eqnarray*}
\sum_{j=0}^{d-1}\sum_{h=d-j-1}^{n-j-1}{{n-j-1}\choose h}{h\choose{d-1-j}}=\\
\sum_{j=0}^{d-1}\sum_{h=d-j-1}^{n-j-1}{{n-j-1}\choose {n-d}}{{n-d}\choose{n-1-j-h}}=\\
\sum_{j=0}^{d-1}{{n-j-1}\choose {n-d}}2^{n-d}={n\choose{d-1}}2^{n-d}.
\end{eqnarray*}
\end{proof}

\begin{Remark}\label{rem:intro} Following Theorem \ref{thm:degree}, the coefficient of $h^{n-d}v^{d-1}w^{m-1}$ in the expansion of
	 the polynomial
$$ (\Sigma_{j=0}^{n-1}(w+h)^{n-1-j}v^j)(\Sigma_{j=0}^{m-1}(v+h)^{m-1-j}w^j)$$
is computed as (\ref{eq:degpairs}) by Newton expansion.
\end{Remark}

\begin{Lemma}\label{lem:eig}
Let $A$ be a $n\times m$ matrix with $n\le m$. $A$ has a singular pair $(v,w)$ with $v\in L$
if and only if $AA^t$ has an eigenvector in $L$.
\end{Lemma}
\begin{proof} If $Aw=\lambda_1 v$ and $A^tv=\lambda_2 w$ it follows $AA^tv=\lambda_1\lambda_2 v$.
Conversely, let $AA^tv=\lambda^2 v$. If $\lambda\neq 0$ pose $w=\lambda^{-1}A^tv$.
If $\lambda=0$, in case $A^tv\neq 0$ we may still pose $w=A^tv$.
In case $A^tv=0$, since $n\le m$ we have $\ker A\neq 0$ then with both $\lambda_i=0$ any $w\in\ker A$ works.
\end{proof} 

\begin{Example}\label{exa:necass} The assumption $n\le m$ is necessary in Lemma \ref{lem:eig}.
Let $A=\begin{pmatrix}1\\ 0\end{pmatrix}$ and $\left(\begin{pmatrix}v_1\\v_2\end{pmatrix},w\right)$
is a singular pair we have $A^tv=v_1$, $Aw=\begin{pmatrix}w\\ 0\end{pmatrix}$.
Set $L=\langle\begin{pmatrix}0\\ 1\end{pmatrix}\rangle$, so $AA^t$ has a eigenvector in $L$
but $Aw\in L$ only if $w=0$.
\end{Example}

\begin{Example}\label{exa:lambdaequal}If $A$ has a singular pair $(v,w)$ with $v\in L$   then
$v$ is eigenvector of $AA^t$, so $AA^t$ has an eigenvector in $L$. The converse is true
on real numbers but it is false on complex numbers if we wish the additional requirement $\lambda_1=\lambda_2$.
A simple counterexample is as follows. Let $A=\begin{pmatrix}1&\sqrt{-1}\\1& \sqrt{-1}\end{pmatrix}$ and
let $L=\langle \begin{pmatrix}1\\ 0\end{pmatrix}\rangle$. The matrix $A$ has no singular pair $(v,w)$ with $v\in L$,
but $AA^t=0$ so $AA^t$ has an eigenvector in $L$.
\end{Example}

\begin{Proposition} Let $d=1$, so that $L=\langle v_0\rangle$ is a line in $V$. Let $C$ be a $(n-1)\times n$ matrix such that $L=\{v|Cv=0\}$.
Then $${\kappa}_{1,n,m}=\{A|C(AA^t) v_0=0, \mathrm{rk}(CA)\le m-1\}.$$
Note the condition on $\mathrm{rk}(CA)$ is empty when $n\le m$.
\end{Proposition}
\begin{proof}
The pair $v_0\otimes(A^tv_0)$ is a singular pair if $(C A)(A^t v_0)=0$, in order to be a nonzero pair we have to exclude the case when $n>m$ and $CA$ has maximal rank $m$.
\end{proof}
We list in Table 1 the first cases of degree of $\kappa_{d,n,m}(L)$ for $n\ge m$ and its singularities.
The singular locus of $\kappa_{d,n,m}(L)$ contains
the closure of the set of matrices having at least two distinct singular pairs with first component on $L$. The table shows that
this containment is strict in many cases, contrary
to the eigenvector case in \cite[Theorem 4.4]{OS}, where equality holds.

{\tiny \begin{table}
	\begin{tabular}{|l|c|c|c|}
		\hline
		d, n, m & $\deg \kappa_{d,n,m}(L)$ &  generators of ideal & singular locus \\
		\hline
		(1,2,2) & 2 & one quadric & $ \varnothing$ \\
		(1,3,2) & 3 & three quadrics & $ \varnothing$ \\
		(2,3,2) & 4 & one quartic &$\{A|ker(A)\supset L\}\cup \{A|ker(A)\supset L^\perp\}\cup (Q_0\times\pp^2)$ \\
		(1,3,3) & 4 & two quadrics & $\{A|ker(A)\supset L, rk(A)\le 1\}$ \\
		(2,3,3) & 6 & one sextic & $\mathrm{codim}=2, \deg=4$ \\
		(1,4,2) & 4 & six quadrics & $ \varnothing$  \\
		(2,4,2) & 6 & one quadric and three quartics &$\{A|ker(A)\supset L,\mathrm{rk}(A)\le 1\}\cup \{A|ker(A)\supset L^\perp\}\cup (Q_0\times\pp^3)$  \\
		(3,4,2) & 4 & one quartic & analogous to $(3,2,2)$ \\
		(1,4,3) & 7 &three quadrics and one cubic&$\mathrm{codim}=7, \deg=42$ \\
		(2,4,3) & 13 & one quartic, two quintics, three sextics &$\mathrm{codim}=4, \deg=40$ \\
		(3,4,3) & 9 & one generator of degree nine & $\mathrm{codim}=2, \deg=13$ \\
			\hline
	\end{tabular}
	\caption{Description of Kalman variety for small values of $(d, n, m)$ }
	\label{Table 1: Values of degree}
\end{table} }

\begin{Remark}
In the matrix case $d=2$, equations for $ \kappa^{s}_{d,n,m}(L)$ were found in \cite{Sam, Hu}, even in the more general case of matrix eigenvectors.
It should be interesting to extend that study to the tensor case.
\end{Remark}


\addcontentsline{toc}{chapter}{Bibliography} \noindent
\end{document}